\documentclass[11pt]{amsart}

\usepackage[utf8]{inputenc}
\usepackage[T1]{fontenc}
\usepackage[english]{babel}
\usepackage{amsmath,amsthm}
\usepackage{ae}
\usepackage{icomma}
\usepackage{units}
\usepackage{color}
\usepackage{graphicx}
\usepackage{bbm}
\usepackage{caption}
\usepackage{array}
\usepackage[hmarginratio=1:1]{geometry}
\usepackage[hyphens]{url}
\usepackage[pdfpagelabels=false,hidelinks]{hyperref}
\usepackage{mathrsfs}
\usepackage{amssymb}
\usepackage{placeins} 
\usepackage{tikz}
\usepackage{float} 
\usepackage[nodayofweek]{datetime}
\usepackage{enumitem}
\usepackage[numbers, sort]{natbib}

\newcommand{\R}{\ensuremath{\mathbb{R}}}
\newcommand{\K}{\ensuremath{\mathcal{K}}}
\DeclareMathOperator{\dist}{\textnormal{dist}}
\renewcommand{\S}{\ensuremath{\mathbb{S}}}
\newcommand{\wOmega}{\ensuremath{\widetilde{\Omega}}}
\newcommand{\hOmega}{\ensuremath{\widehat{\Omega}}}
\DeclareMathOperator{\reg}{reg}
\DeclareMathOperator{\regn}{\mathcal{U}}

\newtheorem{theorem}{Theorem}[section]

\newtheorem{lemma}[theorem]{Lemma}
\newtheorem{remark}[theorem]{Remark}

\newtheorem{corollary}[theorem]{Corollary}
\numberwithin{theorem}{section}
\numberwithin{definition}{section}

\newcommand{\thmref}[1]{Theorem~\ref{#1}}

\begin{document}

\title{A bound for the perimeter of inner parallel bodies}

\author{Simon Larson}
\address{Department of Mathematics, Royal Institute of Technology, SE-10044 Stockholm, Sweden}
\email{simla@math.kth.se}



\begin{abstract}
We provide a sharp lower bound for the perimeter of the inner parallel sets of a convex body $\Omega$. The bound depends only on the perimeter and inradius $r$ of the original body and states that
\[|\partial\Omega_t| \geq \Bigl(1-\frac{t}{r}\Bigr)^{n-1}_+ |\partial \Omega|.\]
In particular the bound is independent of any regularity properties of $\partial\Omega$. As a by-product of the proof we establish precise conditions for equality. The proof, which is straightforward, is based on the construction of an extremal set for a certain optimization problem and the use of basic properties of mixed volumes.

This is a revised version of the paper published in J.\ Funct.\ Anal. (2016) where an error is addressed in accordance with a corrigendum to appear in J.~Funct.\ Anal. The main result of the paper remains the same but an error in Lemma~2.1 has been corrected and the subsequent proofs have been adapted accordingly.
\end{abstract}

\maketitle


\section{Introduction}
Given a convex domain $\Omega \subset \R^n$ we consider the family of its \emph{inner parallel sets}. We denote by $\Omega_t$ the inner parallel set at distance $t\geq0$, which is defined by 
\begin{equation*}
	\Omega_t=\{x\in\Omega : \dist(x, \Omega^c)\geq t\} = \Omega \sim tB.
\end{equation*}
Here $B$ is the unit ball in $\R^n$ and $\sim$ denotes the \emph{Minkowski difference}; a precise definition is given in Section~\ref{sec:NotationPrel}. Correspondingly, the \emph{outer parallel set} at distance $t\geq 0$ is the set 
\begin{equation*}
	\{x\in\R^n : \dist(x, \Omega)\leq t\} = \Omega + tB,
\end{equation*}
where $+$ denotes the \emph{Minkowski sum}. In this paper we provide a lower bound for the perimeter of $\Omega_t$ in terms of the perimeter of $\Omega$. 

An important result in the theory of outer parallel sets is the so-called \emph{Steiner formula}
\begin{equation}\label{eq:SteinerFormula}
	|\Omega + tB| = \sum_{i=0}^n \binom{n}{i} t^i W_i(\Omega),
\end{equation}
where coefficients $W_i$ of the polynomial are the \emph{quermassintegrals} of $\Omega$, which are a special case of mixed volumes (see Section~\ref{sec:NotationPrel}). The set of quermassintegrals contains several important geometric quantities: for instance we have that $W_0(\Omega)=|\Omega|$ and $nW_1(\Omega)=|\partial \Omega|$. There are analogous formulae to~\eqref{eq:SteinerFormula}, called the \emph{Steiner formulae}~\cite{Schneider1}, that express the value of the $i$-th quermassintegral of $\Omega+tB$ in terms of $W_j(\Omega)$, for $i\leq j\leq n$. The Steiner formula appears not only in convex geometry, and important applications may be found in Federer's work on curvature measures in geometric measure theory (see~\cite{Federer}) and Weyl's tube formula in differential geometry (see~\cite{Weyl}). 

For inner parallel sets there is, in general, no counterpart to the Steiner formula. Matheron conjectured in~\cite{Matheron} that the volume of a Minkowski difference is bounded from below by the \emph{alternating Steiner polynomial.}\/ If we restrict our attention to inner parallel sets he conjectured that
\begin{equation*}
	|\Omega \sim t B| \geq \sum_{i=0}^n \binom{n}{i} (-t)^i W_i(\Omega).
\end{equation*}
The precise conjecture was a more general statement where $B$ is replaced by a general convex body and the quermassintegrals are replaced by mixed volumes. However, the conjecture was proved to be false by Hern{\'a}ndez Cifre and Saor{\'{\i}}n in~\cite{HernandezSaorin2}. 
\FloatBarrier

In addition to the lack of a Steiner-type formula, the Minkowski difference is far from being as well behaved as the Minkowski sum. In contrast to the Minkowski sum the difference is not a vectorial operation. Moreover, the regularity properties of $\Omega\sim tB$ may be very different from those of $\Omega$. Both of these properties are demonstrated in Figure~\ref{fig:InnerParallelSet}. Nonetheless, the theory of inner parallel sets is rich and has several beautiful applications in both convex geometry and analysis (see for instance~\cite{Bol1943, Brannen, LinkeSaorin, Sangwine-Yager, Sangwine-Yager3}).

In~\cite{HernandezSaorin1} the authors prove bounds for the quermassintegrals of inner parallel sets in a more general setting than that described above. Instead of considering the sets $\Omega\sim tB$, $t\geq 0$, they consider $\Omega\sim tE$ for some convex set $E$. The inequalities obtained in this paper are closely related to those in~\cite{HernandezSaorin1}, and using similar techniques the results here could, at least in some sense, be generalized to the same setting. However, in such generalizations the connection to the spectral theoretic applications that motivated the work in this paper is lost.

\begin{figure}[h]
	\centering
		
\begin{tikzpicture}[scale=2.8/4]


	\coordinate (a1) at (0,0);
	\coordinate (a2) at (0,3);
	\coordinate (a3) at (4,3);
	\coordinate (a4) at (4,0);
	\node at (-0.3,3.1) {{\,}};
	\draw[thick] (a1)--(a2)--(a3)--(a4)--(a1);

	\coordinate (ai1) at (15/26,15/26);
	\coordinate (ai2) at (15/26,3-15/26);
	\coordinate (ai3) at (4-15/26,3-15/26);
	\coordinate (ai4) at (4-15/26,15/26);
  	\draw[thick, dashed] (ai1)--(ai2)--(ai3)--(ai4)--(ai1);

	\draw[dashed] (ai1) circle [radius=15/26];
	\draw[<->] (ai1)--(0,15/26);

	\node at (16/52,15/26+6.5/26) {{$t$}};
	\node at (2,1.5) {$\widetilde\Omega_t$};
	\node at (3.72,1.53) {$\widetilde\Omega$};


	\coordinate (a1) at (5,15/26);
	\coordinate (a2) at (5,3-15/26);
	\coordinate (a3) at (5+15/26,3);
	\coordinate (a4) at (5+4-15/26,3);
	\coordinate (a5) at (5+4,3-15/26);
	\coordinate (a6) at (5+4,15/26);
	\coordinate (a7) at (5+4-15/26,0);
	\coordinate (a8) at (5+15/26,0);		
	\draw[thick] (a1)--(a2);
	\draw[thick] (a3)--(a4);
	\draw[thick] (a5)--(a6);
	\draw[thick] (a7)--(a8);
	\draw [thick] (a1) arc [radius=15/26, start angle=-180, end angle= -90];
	\draw [thick] (a2) arc [radius=15/26, start angle=180, end angle= 90];
	\draw [thick] (a4) arc [radius=15/26, start angle=90, end angle= 0];
	\draw [thick] (a7) arc [radius=15/26, start angle=-90, end angle= 0];

	\coordinate (ai1) at (5+15/26,15/26);
	\coordinate (ai2) at (5+15/26,3-15/26);
	\coordinate (ai3) at (5+4-15/26,3-15/26);
	\coordinate (ai4) at (5+4-15/26,15/26);
  	\draw[thick, dashed] (ai1)--(ai2)--(ai3)--(ai4)--(ai1);

	\draw[dashed] (ai1) circle [radius=15/26];
	\draw[<->] (ai1)--(5,15/26);

	\node at (5+16/52,15/26+6.5/26) {{$t$}};
	\node at (5+2,1.5) {$\widehat\Omega_t$};
	\node at (5+3.72,1.53) {$\widehat\Omega$};


	\def \a {3};
	\def \b {1};
	\def \t {15/26};
	\def \angle {2.65*pi/4}
	\def \thetaS {0.51};
	\def \thetaE {2.63};

	\draw[thick, domain=0:6.29, samples=100]   plot ({10+\a+\a*cos(\x r)},{1.5+\b*sin(\x r)})   node[right] {};

	\draw[thick, dashed, domain=\thetaS:\thetaE, samples=100] 
	plot ({10+\a+\a*cos(\x r)-\t*\b*cos(\x r)/sqrt(\a^2*sin(\x r)^2+\b^2*cos(\x r)^2)},
	{1.5+\b*sin(\x r)-\t*\a*sin(\x r)/sqrt(\a^2*sin(\x r)^2+\b^2*cos(\x r)^2)})   node[right] {};

	\draw[thick, dashed, domain=-\thetaS:-\thetaE, samples=100] 
	plot ({10+\a+\a*cos(\x r)-\t*\b*cos(\x r)/sqrt(\a^2*sin(\x r)^2+\b^2*cos(\x r)^2)},
	{1.5+\b*sin(\x r)-\t*\a*sin(\x r)/sqrt(\a^2*sin(\x r)^2+\b^2*cos(\x r)^2)})   node[right] {};

	\draw[dashed] ({10+\a+\a*cos(\thetaE r)-\t*\b*cos(\thetaE r)/sqrt(\a^2*sin(\thetaE r)^2+\b^2*cos(\thetaE r)^2)},1.5) circle [radius={\t-1/45}];

	\draw[<->] ({10+\a+\a*cos(\thetaE r)-\t*\b*cos(\thetaE r)/sqrt(\a^2*sin(\thetaE r)^2+\b^2*cos(\thetaE r)^2)},1.5)
	--
	({10+\a+\a*cos(\thetaE r)-\t*\b*cos(\thetaE r)/sqrt(\a^2*sin(\thetaE r)^2+\b^2*cos(\thetaE r)^2)+\t*cos(\angle r)},{1.5+\t*sin(\angle r)}) ;

	\node at ({10+\a+\a*cos(\thetaE r)-\t*\b*cos(\thetaE r)/sqrt(\a^2*sin(\thetaE r)^2+\b^2*cos(\thetaE r)^2)+\t/5},1.5+\t/2) {$t$};
	\node at (13.1,1.5) {$\overline\Omega_t$};
	\node at ({10+\a+\a*cos(\thetaS r)-\t*\b*cos(\thetaS r)/sqrt(\a^2*sin(\thetaS r)^2+\b^2*cos(\thetaS r)^2)+\t/2},1.53) {$\overline\Omega$};

\end{tikzpicture}
\caption{The inner parallel sets of some convex bodies in $\R^2$. Note that $\wOmega_t$ is equal to $\widehat\Omega_t$ even though $\wOmega\neq \widehat\Omega$.}
	\label{fig:InnerParallelSet}
\end{figure}

The main result of this paper is an improvement of the following theorem which is obtained in~\cite{vandenBerg} using the Steiner formula to bound the perimeter of outer parallel sets. 

\begin{theorem}[Modified Steiner inequality~\cite{vandenBerg}]\label{thm:vandenBerg}
Let $\Omega$ be a convex domain in $\R^n$ with volume $|\Omega|$ and surface area $|\partial \Omega|$ such that at each point the principal curvatures of $\partial\Omega$ are bounded from above by $1/K$ for some $K>0$. Then for any $t\geq 0$ we have the bound
\begin{equation*} 
	|\partial \Omega_t| \geq |\partial \Omega| \Bigl(1- \frac{n-1}{K}t\Bigr)_{+}.
\end{equation*}
Further, for any $0\leq t<K$ the principal curvatures of $\partial \Omega_t$ are bounded from above by $(K-t)^{-1}.$ 
\end{theorem}

Our study of this problem is motivated by work of Geisinger, Laptev and Weidl in~\cite{LaptevGeisingerWeidl} where they use \thmref{thm:vandenBerg} to obtain bounds on the Riesz eigenvalue means for the Dirichlet Laplacian on a convex domain $\Omega\subset\R^n$. For convex domains in the plane satisfying the inequality
\begin{equation}\label{eq:LapConjecture}
 	|\partial\Omega_t|\geq \Bigl(1- \frac{3t}{\omega}\Bigr)_+|\partial\Omega|,
\end{equation} 
where $\omega$ denotes the width of $\Omega$, the authors further improve these bounds. Moreover, the authors conjecture that~\eqref{eq:LapConjecture} holds for any planar convex domain. In this paper we prove that the bound~\eqref{eq:LapConjecture} holds for any convex set in $\R^2$ and that similar bounds hold in arbitrary dimension. 

\FloatBarrier
We now turn to our main result which is contained in the next theorem. In~\cite{Larson} the techniques of~\cite{LaptevGeisingerWeidl} are combined with this result to obtain further geometrical improvements of Berezin-type bounds for the Dirichlet eigenvalues of the Laplacian on convex domains.
\begin{theorem}\label{thm:MainRin}
	Let $\Omega\subset \R^n$ be a convex domain with inradius $r$. Then, for any inner parallel set $\Omega_t$, $t\geq0$, it holds that
	\begin{equation*}
	  	|\partial \Omega_t| \geq \Bigl(1- \frac{t}{r}\Bigr)^{n-1}_+|\partial\Omega|.
	\end{equation*} 
	Further, equality holds for some $t\in(0, r)$ if and only if $\Omega$ is homothetic to its form body\footnote{The precise definition of the form body of a convex set $\Omega$ will be given in Section~\ref{sec:NotationPrel} (see also~\cite{HernandezSaorin1,Schneider1}).}. If this is the case equality holds for all $t\geq 0$.
\end{theorem}
Using the above theorem and known bounds for the inradius and width of a convex body we are able to conclude that the conjectured inequality~\eqref{eq:LapConjecture} holds and provide the following generalization to higher dimensions.
\begin{corollary}\label{cor:CorMainWidth}
 	Let $\Omega\subset \R^n$ be a convex domain with width $\omega$. Then, for the inner parallel sets of $\Omega$ we have that
 	\begin{align*}
 		|\partial\Omega_t|&\geq \Bigl(1- \frac{2 \sqrt{n}}{\omega}\,t\Bigr)_+^{n-1}|\partial \Omega| \hspace{29pt}\textrm{if }n\textrm{ is odd,}\\[4pt]
 		|\partial\Omega_t|&\geq \Bigl(1- \frac{2(n+1)}{\omega\sqrt{n+2}}\, t\Bigr)_+^{n-1}|\partial \Omega| \quad\textrm{if }n\textrm{ is even.}
 	\end{align*}
 	In both cases equality holds if $\Omega$ is a regular $(n+1)$-simplex.
\end{corollary}

The result developed here is in several aspects an improvement of~\thmref{thm:vandenBerg}. Firstly, the assumptions on $\Omega$ are less restrictive. We require only convexity whilst the earlier result requires the principal curvatures of $\partial\Omega$ to be bounded. Further, by noting that
\begin{equation*}
	\Bigl(1-\frac{t}{K}\Bigr)^{n-1}_+ \geq \Bigl(1-\frac{(n-1)}{K}t\Bigr)_+
\end{equation*}
and that the maximum of the principal curvatures of the boundary of a convex set is always larger than the reciprocal of its inradius one can conclude that~\thmref{thm:MainRin} implies~\thmref{thm:vandenBerg}. We also note that if $t$ is less than the reciprocal of the maximal principal curvature then $\Omega=\Omega_t + tB$. In general the set $\Omega$ cannot be determined from $\Omega_t$ and $t$.


\FloatBarrier
\subsection{Notation and preliminaries}\label{sec:NotationPrel}

Let $\K^n_0$ denote the set of all convex bodies in $\R^n$ that have nonempty interior. Throughout the paper $\Omega$ will belong to $\K^n_0$. Let $B$ denote the closed unit ball in $\R^n$ and let $\S^{n-1}$ denote the corresponding sphere. A closed ball of radius $r$ centred at $x\in \R^n$ is denoted by $B_r(x)$. For notational simplicity we denote both volume and surface measure by $|\cdot|$. This will appear in two forms, the volume of a set $|\Omega|$ and the surface measure of its boundary $|\partial \Omega|$. Further, we will make use of the notation $x_\pm=(|x|\pm x)/2.$

For two sets $K, L\in\K^n_0$ the \emph{Minkowski sum} ($+$) and \emph{difference} ($\sim$) are defined by
\begin{align*}
		K+L &:= \{ x+y : x\in K, y\in L\}, \\
		K\sim L &:= \{ x\in \R^n : x+L \subseteq K\}.
\end{align*}
It is a direct consequence of the definitions that we, as claimed in the introduction, equivalently can define the inner parallel body $\Omega_t$, $t\geq 0$, as $\Omega\sim tB$ \cite{Schneider1}. Similarly the outer parallel body can be written as $\Omega+tB$. 

The \emph{inradius} $r$ of a set $\Omega\in \K^n_0$ is defined as the radius of the largest ball contained in $\Omega$, or equivalently (see for instance \cite{Schneider1}) as
\begin{equation*}
	r = \sup \{\lambda \geq 0 : \Omega\sim \lambda B \neq \emptyset\}.
\end{equation*}

The observation contained in the next lemma is intuitively clear but of central importance in what follows. 
\begin{lemma}\label{lem:InradiusLemma}
	Let $\Omega\in \K^n_0$ have inradius $r_0$. Then, for any $t\in [0, r_0]$ the inradius $r_t$ of $\Omega_t$ satisfies
	\begin{equation*}
	 	r_t = r_0-t.
	 \end{equation*}
\end{lemma}
\begin{proof}
	Let $x_0\in \R^n$ be such that $B_{r_0}(x_0)\subseteq \Omega$. For each $x\in  B_{r_0}(x_0)$ we have that $\dist(x, \partial\Omega)\geq \dist(x, \partial B_{r_0}(x_0))$ and hence
	\begin{equation*}
		\Omega_t \supseteq \bigl(B_{r_0}(x_0)\bigr)_t = B_{(r_0-t)}(x_0).
	\end{equation*}
	We conclude that $r_t\geq r_0-t$. To prove the reverse inequality we observe that for any $x_t\in \R^n$ such that $ B_{r_t}(x_t)\subseteq \Omega_t$ we have that $\dist(B_{r_t}(x_t), \partial\Omega) \geq t$. Which implies that 
	\begin{equation*}
	 	B_{(r_t+t)}(x_t)= B_{r_t}(x_t)+t B \subseteq \Omega,
	 \end{equation*}
	 and consequently $r_0\geq r_t+t$.
\end{proof}

A classic result in convex geometry is that the volume of a Minkowski sum $\lambda_1 K_1+\dots+\lambda_m K_m$ is, for $\lambda_1, \dots, \lambda_m\geq 0$ and $K_1, \dots, K_m\in \K^n_0$, a homogeneous $n$-th degree polynomial in the $\lambda_i$ with positive coefficients (see~\cite{BonnesenFenchel, Schneider1}). That is, we can write
\begin{equation*}
	|\lambda_1 K_1+\dots + \lambda_m K_m| = \sum_{i_1=1}^m\ldots\sum_{i_n=1}^m \lambda_{i_1}\!\!\cdots\lambda_{i_n} W(K_{i_1}, \dots, K_{i_n}),
\end{equation*}
where $W$ is symmetric with respect to its arguments. The $W(K_{i_1}, \dots, K_{i_n})$ are called the \emph{mixed volumes} of $K_{1}, \dots, K_{m}$. In what follows we will use several properties of $W\!$. We list the properties here and for proofs refer to~\cite{BonnesenFenchel, Schneider1}:
\begin{itemize}
	\item $W$ is a symmetric functional on $n$-tuples of sets in $\K^n_0$. 
	\item $W$ is multilinear with respect to Minkowski addition:
	\begin{equation*}
		W(\lambda K+\lambda'K', K_2, \dots, K_{n}) = \lambda W(K, K_2, \dots, K_n)+\lambda' W(K', K_2, \dots, K_n).
	\end{equation*}
	\item $W$ is monotonically increasing with respect to inclusions.
	\item $W$ is invariant under translations in each argument.
	\item The perimeter of $K\in \K^n_0$ is, up to a constant, equal to a mixed volume:
	\begin{equation*}
		|\partial K| = nW(B, K, \dots, K).
	\end{equation*}
\end{itemize}

We will by $h(K, u)$ denote the \emph{support function} of $K \in \K_0^n$ which is defined for any $u\in \R^n$ as
\begin{equation*}
	h(K, u)= \sup_{x\in K} \langle x, u\rangle.
\end{equation*}
The restriction of $h(K, u)$ to $u\in\S^{n-1}$ reduces to the function describing the distance from the origin to the supporting hyperplane of $K$ with normal $u$. In what follows we denote such a supporting hyperplane by $H(K, u)$. We then have the following characterization of the supporting hyperplanes of $K$:
\begin{equation*}
 	H(K, u) = \{ x\in\R^n : \langle x, u\rangle = h(K, u)\}.
\end{equation*} 
The following properties of $h(K, u)$ will be needed later:
\begin{itemize}
 	\item For any $K, L \in \K^n_0$ and $\alpha, \beta>0$ it holds that 
 	\begin{equation*}
 		h(\alpha K+\beta L, u)=\alpha h(K, u)+\beta h(L, u).
 	\end{equation*}
 	\item For any $u\in \S^{n-1}$ and $K, L \in K^n_0$ it holds that
 	\begin{equation*}
 		h(K\sim L, u)\leq h(K, u)-h(L, u).
 	\end{equation*}
 	\item For $x \in \partial(K\sim L)$ there exists a normal vector $u$ of $\partial(K\sim L)$ at $x$ such that
 	\begin{equation*}
 	 	h(K\sim L, u)=h(K, u)-h(L, u).
 	\end{equation*}
\end{itemize}
Proofs of the above properties can be found in~\cite{Schneider1}.

The \emph{width} $\omega$ of $\Omega \in \K^n_0$ is defined as
\begin{equation*}
 	\omega = \inf \{ h(\Omega, u) + h(\Omega, -u) : u\in \S^{n-1}\}.
\end{equation*}
A point $x\in\partial\Omega$ is called \emph{regular} if the supporting hyperplane at $x$ is uniquely defined, that is if there is a unique $u\in \S^{n-1}$ such that
\begin{equation*}
	x \in H(\Omega, u)\cap \Omega.
\end{equation*}
The set of all regular points of $\partial\Omega$ is denoted by $\reg(\Omega)$. We also let $\regn(\Omega)$ denote the set of all outward pointing unit normals to $\partial\Omega$ at points of $\reg(\Omega)$.

We are now ready to define the \emph{form body} $\Omega_*$ of a set $\Omega\in \K^n_0$, which, following~\cite{Schneider1}, is defined by
\begin{equation*}
	\Omega_* = \bigcap_{u\in\regn(\Omega)} \{ x\in \R^n : \langle x, u\rangle \leq 1\}.
\end{equation*}
If $\Omega$ is a polytope then $\Omega_*$ is the polytope that has the same set of normals as $\Omega$, but with each face translated so that it is tangent to the unit ball. If instead the boundary of $\Omega$ is smooth (in which case every point is regular) then $\Omega_*=B$.
\begin{figure}[H]
 	\centering
 	\begin{tikzpicture}[scale=1/6]


	\def \xRef {4.4940};
	\def \yRef {5.5};

	\node at (-9,0) {{\,}};
	\coordinate (a1) at ({0-\xRef},{0-\yRef});
	\coordinate (a2) at ({-4-\xRef},{6-\yRef});
	\coordinate (a3) at ({4-\xRef},{10-\yRef});
	\coordinate (a4) at ({12-\xRef},{8-\yRef});
	\coordinate (a5) at ({8-\xRef},{2-\yRef});
	\draw[thick] (a1)--(a2)--(a3)--(a4)--(a5)--(a1);

 	\coordinate (c3) at ({4.4940-\xRef},{5.5-\yRef});
	
	\coordinate (m1) at ({-2-\xRef},{3-\yRef});
	\coordinate (m2) at ({0-\xRef},{8-\yRef});
	\coordinate (m3) at ({8-\xRef},{9-\yRef});
	\coordinate (m4) at ({10-\xRef},{5-\yRef});
	\coordinate (m5) at ({4-\xRef},{1-\yRef});

	\def \lref {(8.494*0.4471+0.5*0.8944)};
	\draw[dashed] (c3) circle [radius=\lref];
	\draw[<->] (c3)--({-0.4472*\lref + 4.494-\xRef},{5.5+0.8944*\lref-\yRef});
	\node at ({-0.4472*\lref/2 + 4.494+1.2-\xRef},{5.5+0.8944*\lref/2-\yRef+0.1}) {$r$};


	\node at ({-7},{2.8}) {\scalebox{1}{$\Omega$}};

	
	\def \xRef {17};
	\def \yRef {0};
	\draw[dashed] (\xRef,\yRef) circle [radius=\lref];
	\draw[<->] (\xRef,\yRef)--({\xRef-0.4471*\lref},{\yRef+0.8944*\lref});
	\node at ({\xRef-0.4471*\lref/2+1.2},{\yRef+0.8944*\lref/2+0.2}) {$1$};

	\draw[thick, domain={-1.4604*\lref}:{-0.4411*\lref},samples=2] plot ({\xRef+\x},{\yRef-(-0.8321*\x-\lref)/(-0.5547)}) node[right] {}; 
	\draw[thick, domain={-1.4604*\lref}:{-0.1163*\lref},samples=2] plot ({\xRef+\x},{\yRef-(-0.4472*\x-\lref)/(0.8944)}) node[right] {}; 
	\draw[thick, domain={-0.1163*\lref}:{1.6192*\lref},samples=2] plot ({\xRef+\x},{\yRef-(0.2425*\x-\lref)/(0.9701)}) node[right] {}; 
	\draw[thick, domain={1.6192*\lref}:{0.6176*\lref},samples=2] plot ({\xRef+\x},{\yRef-(0.8321*\x-\lref)/(-0.5547)}) node[right] {}; 
	\draw[thick, domain={0.6176*\lref}:{-0.4411*\lref},samples=2] plot ({\xRef+\x},{\yRef-(0.2425*\x-\lref)/(-0.9701)}) node[right] {}; 

	\node at ({\xRef-5},{\yRef+3.85}) {\scalebox{1}{$\Omega_*$}};


	\def \xRef {38};
	\def \yRef {0};

	\draw[dashed] (\xRef,\yRef) circle [radius=\lref];
	\draw[<->] (\xRef,\yRef)--({\xRef-0.4471*\lref},{\yRef+0.8944*\lref});
	\node at ({\xRef-0.4471*\lref/2+1.2},{\yRef+0.8944*\lref/2+0.2}) {$r$};

	\coordinate (a1) at ({-4.4940+\xRef},{-5.5+\yRef});
	\coordinate (a2) at ({-4-4.4940+\xRef},{6-5.5+\yRef});
	\coordinate (a3) at ({4-4.4940+\xRef},{10-5.5+\yRef});
	\coordinate (a4) at ({12-4.4940+\xRef},{8-5.5+\yRef});
	\coordinate (a5) at ({8-4.4940+\xRef},{2-5.5+\yRef});
	\draw[thick] (a2)--(a3)--(a4);

	\coordinate (m1) at ({-2-4.494+\xRef},{3-5.5+\yRef});
	\coordinate (m2) at ({0-4.4940+\xRef},{8-5.5+\yRef});
	\coordinate (m3) at ({8-4.4940+\xRef},{9-5.5+\yRef});
	\coordinate (m4) at ({10-4.4940+\xRef},{5-5.5+\yRef});
	\coordinate (m5) at ({4-4.4940+\xRef},{1-5.5+\yRef});

	\draw[thick,rounded corners=5mm] (a2)--(a1)--(m5);

	\draw[thick,rounded corners=4mm] (m5)--(a5)--(a4);

  	\node at ({-7+\xRef},{3+\yRef}) {\scalebox{1}{$\widetilde\Omega$}};


  	\def \xRef {55};
	\def \yRef {0};

	\draw[dashed] (\xRef,\yRef) circle [radius=\lref];
	\draw[<->] (\xRef,\yRef)--({\xRef-0.4471*\lref},{\yRef+0.8944*\lref});
	\node at ({\xRef-0.4471*\lref/2+1.2},{\yRef+0.8944*\lref/2+0.2}) {$1$};

	\draw[thick, domain={-1.4604*\lref}:{-0.1163*\lref},samples=2] plot ({\xRef+\x},{\yRef-(-0.4472*\x-\lref)/(0.8944)}) node[right] {}; 
	\draw[thick, domain={-0.1163*\lref}:{1.6192*\lref},samples=2] plot ({\xRef+\x},{\yRef-(0.2425*\x-\lref)/(0.9701)}) node[right] {};

	\draw[thick, domain={-1.4604*\lref}:{-0.8321*\lref},samples=2] plot ({\xRef+\x},{\yRef-(-0.8321*\x-\lref)/(-0.5547)}) node[right] {}; 
	
	\draw[thick, domain={1.6192*\lref}:{0.8321*\lref},samples=2] plot ({\xRef+\x},{\yRef-(0.8321*\x-\lref)/(-0.5547)}) node[right] {}; 
	
	\draw[thick, domain={-0.5879}:{-2.5537}, samples=20] plot ({\xRef+\lref*cos(\x r)},{\yRef+\lref*sin(\x r)}) node[right] {};


	\node at ({\xRef-5},{\yRef+4.1}) {\scalebox{1}{$\widetilde\Omega_*$}};
  	
\end{tikzpicture}
 	\caption{The form body of two convex sets in $\R^2$.}
 	\label{fig:OmegaT}
\end{figure} 
The following lemma will be needed in our main proof and is an almost direct consequence of the definitions of $\Omega_*$ and $r$ combined with the fact that almost every point of $\partial\Omega$ is regular. 
\begin{lemma}\label{lem:InclusionFormBody}
 	Let $\Omega \in \K^n_0$ have inradius $r$. Then there exists $x\in \R^n$ such that $x+r\Omega_* \subseteq \Omega$.
\end{lemma}


\FloatBarrier
\section{Proof of the main result}

The idea of the proof is as follows: Given a set $\Omega\in \K_0^n$ and $t\geq0$ we construct a convex set $\wOmega$ such that $\wOmega_t=\Omega$ and $|\partial \wOmega|\geq |\partial\widehat\Omega|$ for any other set $\widehat{\Omega}\in \K_0^n$ satisfying $\widehat\Omega_t=\Omega$. If we can prove \thmref{thm:MainRin} for such $\wOmega$ it clearly holds also for any other convex set satisfying $\widehat\Omega_t=\Omega$. Since the choice of $\Omega$ and $t$ was arbitrary this completes the proof.

We begin by constructing the set $\wOmega$. In the case where $\Omega$ is a polygon this problem has the fairly intuitive solution that $\wOmega$ is the polygon with the same faces as $\Omega$, only moved a distance $t$ along their outward pointing normals. The following lemma tells us that a generalization of this intuitive solution actually works for any possible $\Omega$.

\begin{lemma}\label{lemma:MaxProblem}
	Let $\Omega\in \K^n_0$ and let $\Omega_*$ denote its form body. Then, for any $t\geq0$ the maximization problem 
	\begin{equation*}
		\max\{|\partial \widehat\Omega| : \widehat\Omega \in \K^n_0,\; \widehat\Omega_t=\Omega\} 
	\end{equation*}
	is solved by
	\begin{equation*}
		\widetilde \Omega = \bigcap_{u\in \mathcal{U}(\Omega)}\{x\in \R^n: \langle x, u\rangle \leq h(\Omega, u)+t\}.
	\end{equation*}
\end{lemma}
\begin{remark}
	\textnormal{In the original version of this paper the set $\widetilde\Omega$ was erroneously identified as $\Omega +t\Omega_*$. The error was pointed out in~\cite{cifre_isoperimetric_2019} and addressed in the corrigendum~\cite{corrigendum}.}
\end{remark}

\begin{proof}
	Recalling the properties of the support function, we have that for any $x\in \partial\Omega = \partial (\hOmega\sim tB)$ there exists a $u\in \S^{n-1}$ normal to $\partial \Omega$ at $x$ such that
	\begin{equation*}
		h(\hOmega \sim tB, u)=h(\hOmega, u)- h(t B, u)=h(\hOmega, u)-t.
	\end{equation*}
	Rearranging this we find for all $u\in\regn(\Omega)$ that $h(\hOmega, u)=h(\Omega, u)+t$. Therefore, it follows that
	\begin{equation*}
		\hOmega \subseteq \bigcap_{u\in \regn(\Omega)}\{ x\in \R^n : \langle x, u\rangle \leq h(\Omega, u)+t\} = \widetilde \Omega.
	\end{equation*}
	Since the perimeter is increasing under inclusion of convex sets, we conclude that $|\partial\hOmega|\leq |\partial \widetilde\Omega|$. 

	What remains to complete the proof is to show that $\widetilde\Omega$ is an admissible set in the above maximization problem. That $\Omega \subseteq \widetilde\Omega_t$ follows from the argument above so we only need to establish the opposite inclusion. Let $x\in \reg(\Omega)$. Then, with $u$ being the unique normal to $\partial\Omega$ at $x$, we have that
	\begin{equation*}
		h(\widetilde \Omega, u)=h(\Omega, u)+t.
	\end{equation*}
	Since a convex body can be written as the intersection of its supporting half-spaces we conclude that $x+t u \in \overline{\widetilde\Omega^c}$ implying that $\dist(x, \partial \widetilde\Omega)\leq t$. Combining this with the inclusion of $\Omega$ in $\widetilde\Omega_t$ we find that $\reg(\Omega)\in \partial\widetilde\Omega_t$. Since almost every point of $\partial\Omega$ is regular the statement follows.
\end{proof}

We are now ready to prove \thmref{thm:MainRin}. Let $t\geq 0$ and let $\Omega \in \K^n_0$ have inradius $r$. By the above lemma we have, for any convex body $\hOmega$ such that $\hOmega_t=\Omega$, the bound $|\partial \hOmega|\leq |\partial \widetilde\Omega|$ with $\widetilde\Omega$ as in Lemma~\ref{lemma:MaxProblem}, and by Lemma~\ref{lem:InradiusLemma} any such $\hOmega$ has the inradius $\hat{r}=r+t$. Thus it is sufficient to prove that
\begin{equation*}
	|\partial\Omega|\geq \Bigl(1- \frac{t}{ \hat r}\Bigr)_+^{n-1}|\partial\widetilde\Omega|,
\end{equation*}
or equivalently
\begin{equation}\label{eq: goal}
	\Bigl(1+ \frac{t}{r}\Bigr)^{n-1}|\partial\Omega|\geq |\partial\widetilde\Omega|
\end{equation}
We shall argue that $\widetilde\Omega \subseteq \bigl(1+\frac{t}{r}\bigr)\Omega$ from which~\eqref{eq: goal} follows by monotonicity of the perimeter under inclusion of convex bodies and scaling. Assume without loss that $B_r(0)\subseteq \Omega$, then $h(\Omega, u)\geq r$ for all $u\in \S^{n-1}$. Therefore,
\begin{align*}
 	\widetilde\Omega &= \bigcap_{u\in \mathcal{U}(\Omega)}\{x\in \R^n: \langle x, u\rangle \leq h(\Omega, u)+ t\}\\
 	&\subseteq 
 	\bigcap_{u\in \mathcal{U}(\Omega)}\Bigl\{x\in \R^n: \langle x, u\rangle \leq \Bigl(1+ \frac{t}{r}\Bigr)h(\Omega, u)\Bigr\} = \Bigl(1+ \frac{t}{r}\Bigr)\Omega,
 \end{align*} 
which proves the claimed inclusion. Moreover, it is clear that equality holds if and only if $h(\Omega, u)=r$ for all $u\in \mathcal{U}(\Omega)$ that is when $\Omega$ is homothetic to its form body. This completes the proof of Theorem~\ref{thm:MainRin}.

Deducing Corollary~\ref{cor:CorMainWidth} is simply a matter of applying the following theorem due to Steinhagen~\cite{Steinhagen}. We note that this theorem appeared in the case of planar convex bodies in earlier work by Blaschke~\cite{Blaschke}, and this simpler case is sufficient for proving the inequality conjectured in~\cite{LaptevGeisingerWeidl}.

\begin{theorem}[Steinhagen's inequality~\cite{Steinhagen}]
	Let $\Omega\in \K^n_0$ have inradius $r$ and width $\omega$. Then the following two-sided inequality holds:
	\begin{align*}
		2r\leq \omega &\leq 2\sqrt{n}\, r \hspace{31pt}\textrm{if n is odd},\\
		2r\leq \omega &\leq \frac{2(n+1)}{\sqrt{n+2}} r \quad\textrm{ if n is even.}
	\end{align*}
	The lower bound is attained if $\Omega$ is a ball, and the upper bound is attained if $\Omega$ is a regular $(n+1)$-simplex.
\end{theorem}


\begin{thebibliography}{19}

\def\myarXiv#1#2{\href{http://arxiv.org/abs/#1}{\texttt{arXiv:#1\, [#2]}}}

\bibitem{Blaschke}
W.~Blaschke, \emph{\"{U}ber den gr\"o\ss{t}en {K}reis in einer konvexen
  {P}unktmenge}, Jahresber. Dtsch. Math.-Ver. \textbf{23} (1914), 369--374.

\bibitem{Bol1943}
G.~Bol, \emph{Beweis einer vermutung von {H}. {M}inkowski}, Abh. Math. Semin.
  Univ. Hambg. \textbf{15} (1943), no.~1, 37--56.

\bibitem{BonnesenFenchel}
T.~Bonnesen and W.~Fenchel, \emph{Theorie der konvexen K\"o{r}per}, Springer, Berlin, 1934. English translation: \emph{Theory of convex bodies}, Edited by L. Boron, C. Christenson and B. Smith, BCS Associates, Moscow, ID, 1987.

\bibitem{Brannen}
N.~S. Brannen, \emph{The {W}ills conjecture}, Trans. Amer. Math. Soc.
  \textbf{349} (1997), no.~10, 3977--3987.

\bibitem{Federer}
H.~Federer, \emph{Curvature measures}, Trans. Amer. Math. Soc. \textbf{93}
  (1959), 418--491.

\bibitem{LaptevGeisingerWeidl}
L.~{Geisinger}, A.~{Laptev}, and T.~{Weidl}, \emph{{Geometrical versions of
  improved {B}erezin-{L}i-{Y}au Inequalities}}, J. Spectr. Theory
  \textbf{1} (2011), 87--109.
  
\bibitem{HernandezSaorin2}
M.~A. Hern{\'a}ndez~Cifre and E.~Saor{\'{\i}}n, \emph{{O}n the volume
  of inner parallel bodies}, Adv. Geom. \textbf{10} (2010), no.~2, 275--286.

\bibitem{HernandezSaorin1}
M.~A. Hern\'andez~Cifre and E.~Saor\'in~G\'omez, \emph{On inner parallel bodies and quermassintegrals}, Israel J. Math.
  \textbf{177} (2010), 29--47.

\bibitem{cifre_isoperimetric_2019}
M.~A. Hern\'andez~Cifre and E.~Saor\'in~G\'omez, \emph{Isoperimetric relations
  for inner parallel bodies}, \myarXiv{1910.05367}{math.MG} (2019).

\bibitem{Larson}
S.~{Larson}, \emph{{On the remainder term of the Berezin inequality on a convex domain}}, Proc. Amer. Math. Soc. \textbf{145} (2017), no.~5, 2167-2181.

\bibitem{corrigendum}
S.~{Larson}, \emph{{Corrigendum to "A bound for the perimeter of inner parallel bodies" [J.\ Funct. Anal. 271 (3) (2016) 610-619)]}}, J.\ Funct.\ Anal.\ (published online).
    
\bibitem{LinkeSaorin}
E.~Linke and E.~Saor{\'{\i}}n~G{\'o}mez, \emph{Decomposition of polytopes using
  inner parallel bodies}, Monatsh. Math. \textbf{176} (2015), no.~4, 575--588.

\bibitem{Matheron}
G.~Matheron, \emph{La formule de {S}teiner pour les \'erosions}, J. Appl.
  Probab. \textbf{15} (1978), no.~1, 126--135.

\bibitem{Sangwine-Yager}
J.~R. Sangwine-Yager, \emph{Inner parallel bodies and geometric inequalities},
  ProQuest LLC, Ann Arbor, MI, 1978, Thesis (Ph.D.)--University of California,
  Davis.

\bibitem{Sangwine-Yager3}
J.~R. Sangwine-Yager, \emph{A {B}onnesen-style inradius inequality in {$3$}-space}, Pacific
  J. Math. \textbf{134} (1988), no.~1, 173--178.

\bibitem{Schneider1}
R.~Schneider, \emph{{C}onvex bodies: the {B}runn-{M}inkowski theory}, {S}econd
  {E}xpanded ed., Encyclopedia of Mathematics and its Applications, vol.~151,
  Cambridge Univ. Press, Cambridge, 2014.

\bibitem{Steinhagen}
P.~Steinhagen, \emph{\"{U}ber die gr\"o\ss te {K}ugel in einer konvexen
  {P}unktmenge}, Abh. Math. Semin. Univ. Hambg. \textbf{1} (1922), no.~1,
  15--26.

\bibitem{vandenBerg}
M.~van~den Berg, \emph{A uniform bound on trace{$\,(e^{t\Delta })$} for convex
  regions in {${\R}^{n}$} with smooth boundaries}, Comm. Math. Phys.
  \textbf{92} (1984), no.~4, 525--530.

\bibitem{Weyl}
H.~Weyl, \emph{On the {V}olume of {T}ubes}, Amer. J. Math. \textbf{61} (1939),
  no.~2, 461--472.

\end{thebibliography}
\end{document}